\documentclass[12pt,leqno]{amsart}
\usepackage{amssymb,amsmath,amsthm,amssymb}
\usepackage{color}
\usepackage[normalem]{ulem}
\usepackage[bookmarksopen,bookmarksdepth=3,colorlinks,citecolor=red,pagebackref,hypertexnames=true]{hyperref}
\usepackage[msc-links,nobysame,non-sorted-cites, initials]{amsrefs}
\usepackage[inline]{enumitem}
\usepackage{mathtools}
\mathtoolsset{showonlyrefs}

\definecolor{darkblue}{RGB}{0,0,160}

\usepackage[lining]{libertine}   
\usepackage{cabin}
\usepackage[libertine]{newtxmath}
\usepackage[utf8]{inputenc}
\usepackage[T1]{fontenc}

\def\eps{\varepsilon}
\def\d{{\rm d}}
\def\R {\mathbb{R}}
\def\N {\mathbb{N}}

\def\supp {{\mathrm{supp}\,}}

\def\C {{\mathrm C}}
\def\B {{\mathcal B}}
\def\D {{\mathcal D}}
\def \l {\langle}
\def \r {\rangle}
\def\bb{\mathfrak{b}}
\def\Sy {{\mathsf{Sy}}}
\newcommand{\lip}{\langle}
\newcommand{\rip}{\rangle}

\numberwithin{equation}{section}
\numberwithin{counter2}{section}
\newtheorem{proposition}[subsection]{Proposition}
\newtheorem{theorem}[counter]{Theorem}

\newtheorem{lemma}[subsection]{Lemma}

\theoremstyle{definition}
\newtheorem{definition}[subsection]{Definition}
\newtheorem*{remark*}{Remark}
\newtheorem*{example*}{Example}
\newtheorem*{warn*}{A word of warning}

\newtheorem{remark}[subsection]{Remark} 
\theoremstyle{plain}

\newcommand{\ip}[2]{\left\langle #1,#2 \right\rangle}
\newcommand{\abs}[1]{\left\vert #1\right\vert}
\newcommand{\norm}[1]{\left\Vert #1\right\Vert}
\newcommand{\ep}{\varepsilon}
\newcommand{\BMO}{\mathrm{BMO}(\mathbb R^d)}
\newcommand{\CMO}{\mathrm{CMO}(\mathbb R^d)}
\newcommand{\Schw}{{\mathcal{S}}}
\newcommand{\varu}{\vartheta}
\newcommand{\osc}{\mathsf{osc}}

\title{Multilinear Wavelet Compact T(1) Theorem}
\author[A. Fragkos]{Anastasios Fragkos}
\address[A. Fragkos]{School of Mathematics, Georgia Institute of Technology, Atlanta, GA 30332, USA}
		\email{\href{mailto:anastasiosfragkos@gatech.edu}{\textnormal{anastasiosfragkos@gatech.edu}}}
\author[A. W. Green]{A. Walton Green}
\thanks{A. W. Green's research supported in part by NSF-DMS-2202813.}
\author[B. D. Wick]{Brett D. Wick}
\thanks{B. D. Wick's research supported in part by NSF-DMS-2000510, NSF-DMS-2054863 as well as ARC DP 220100285.}
\address[A. W. Green and B. D. Wick]{Department of Mathematics, Washington University in Saint Louis\\ \newline \indent 1 Brookings Drive, Saint Louis, MO 63130, USA}
\email{\href{mailto:awgreen@wustl.edu}{\textnormal{awgreen@wustl.edu}},\href{mailto:bwick@wustl.edu}{\textnormal{bwick@wustl.edu}}}

\begin{document}

\begin{abstract}
We prove a wavelet $T(1)$ theorem for compactness of multilinear Calder\'{o}n-Zygmund (CZ) operators. Our approach characterizes compactness in terms of testing conditions and yields a representation theorem for compact CZ forms in terms of wavelet and paraproduct forms that reflect the compact nature of the operator.
\end{abstract}	

\maketitle

\section{Introduction}
A {continuous} $(m+1)$-linear form $\Lambda$ defined on the $(m+1)$-fold product of the Schwarz space $\Schw(\mathbb R^d)$ is a singular integral form if its off-diagonal kernel satisfies the standard size and smoothness estimates (see Definition \ref{d:si} below). A singular integral form $\Lambda$ is bounded on $L^{p_0}(\mathbb R^d) \times \cdots \times L^{p_{m}}(\mathbb R^d)$ for all
	\begin{equation}\label{e:p} 1< p_j \le \infty, \quad \sum_{j=0}^m \frac 1{p_j} = 1,\end{equation}
if and only if $\Lambda$ is \textit{Calder\'on-Zygmund}, which we take to mean that it satisfies the weak boundedness property (see Definition \ref{d:cz} below) and the following multilinear $T(1)$ condition: There exists $\bb_j \in \BMO$ such that for every $\phi$ in $\Schw(\mathbb R^d)$ with mean zero,
	\begin{equation}\label{e:T1} \Lambda^{*,j}(\phi,1,\ldots,1) = \langle \phi,\bb_j \rangle,\end{equation}
where $\Lambda^{*,j}$ permutes the $0$th and $j$th argument (see \eqref{e:adj} below). When $m=1$, this is the celebrated $T(1)$ theorem of David and Journ\'e \cite{david84} which was extended to $m \ge 2$ by Grafakos and Torres \cite{GT02} (see also the earlier works \cites{christ1987polynomial,coifman-meyer}) and more recently developed in, for example, \cites{DPLiMV2,DPLiMV1,LN19,li2019bilinear,li2014sharp}. 

Our goal here is to prove a $T(1)$ representation theorem for compactness of multilinear singular integral forms. The first distinction between the boundedness problem and compactness problem is that compactness is a property of operators, while boundedness in the reflexive range \eqref{e:p} can be equivalently stated in terms of forms. Accordingly for each $j=0,1,\ldots,m$, we associate to $\Lambda$ the $m$-linear adjoint operators $T^{*,j}$ and transposed forms $\Lambda^{*,j}$ by
	\begin{equation}\label{e:adj} \langle f_0,T^{*,j}(f_1,\ldots,f_m) \rangle = \Lambda^{*,j}(f_0,f_1,\ldots,f_m) = \Lambda(f_j,f_1,\ldots,f_{j-1},f_0,f_{j+1},\ldots,f_m).\end{equation}
If $\Lambda$ is Calder\'on-Zygmund, then we say each $T^{*,j}$ is an $m$-linear Calder\'on-Zygmund operator (CZO). {Often we use the shorthand $T=T^{*,0}$.} 

Compactness of singular integral and related operators has a long history, with contributions from several distinct directions in harmonic analysis and operator theory. In particular, the study of commutators of Calder\'on-Zygmund operators, initiated by Uchiyama in the context of Hankel operators \cite{Uch78}, has developed into a rich literature; see, e.g., \cites{ MR1397362,MR924766, MR1021138} for some foundational work on compactness of commutators and tools used and  \cites{BT13, LL22, MR3834654} for some more recent contributions to the area. 
In another direction, Cordes’ work \cite{Cor75} on compactness of pseudodifferential operators, which lie outside the class of Calder\'on-Zygmund operators, provides sufficient conditions for compactness in terms of the symbol of the pseudodifferential operator. That line of investigation has seen further recent developments \cite{CIRXY24,CST, TX20}.  These results highlight the importance of compactness phenomena in harmonic analysis, though they concern somewhat different mechanisms than those addressed here.  In particular, these lines of inquiry are different from what we explore in this paper, as the behavior in question is connected to a particular symbol (function theory), whereas ours is some property intrinsic to a specific operator (its compactness).

In the linear case, Villarroya in \cite{Vil15} gave a complete characterization of compact Calder\'{o}n-Zygmund operators on $L^2(\R^d)$, which was further developed in \cites{OV,PPV,SVW,Vil19}. Recently, Mitkovski and Stockdale in \cite{MS23} gave a simplified formulation of the $T(1)$ theorem for compactness of Villarroya. More precisely, they showed that a CZO $T$ is compact if and only if $\bb_0$ and $\bb_1$ both belong to $\CMO$ and a vanishing version of the weak boundedness property, called the weak compactness property, is satisfied. {Even more recently, B\'enyi, Li, Oh, and Torres \cite{BLOT24b} have characterized compact CZOs in terms of $L^2$ and $L^\infty$ testing conditions, in the spirit of Stein's formulation of the $T(1)$-theorem \cite{stein-book}.} These $T(1)$ theorems for compactness are complemented by a compact Rubio De Francia theory of extrapolation due to Hyt\"onen and Lappas in \cites{HL2,HL1}, and subsequently developed by others in \cites{COY,CST}. Particularly, in \cite{COY}, Cao, Olivo, and Yabuta extended the bilinear results of \cite{HL2} to the multilinear setting and to the quasi-Banach range, in which case the target space can be $L^r$ with $r > \frac 1m$. However, due to the difficulties of extrapolating to the upper endpoint in the multilinear setting \cites{li2021endpoint,nieraeth2019quantitative}, the results of \cites{COY,HL2} do not consider the case where one (or more) input spaces is $L^\infty(\mathbb R^d)$. We point out that our results below do yield compactness when an input space is $L^\infty(\mathbb R^d)$.

{Our main result is Theorem \ref{thm:all} below which builds on the literature discussed above in the following two noteworthy manners. First, it provides a $T(1)$-type characterization of compact multilinear CZOs. Though extrapolation of compactness had been developed in the multilinear setting as mentioned above, a $T(1)$-type characterization in the vein of Villarroya \cite{Vil15} or the simplified formulation of Mitkovski and Stockdale \cite{MS23} was heretofore unknown. Our characterization is indeed in the simplified form of \cite{MS23} (not relying on compact CZ kernel estimates) while at the same time provides a stronger result than both \cite{BLOT24b} and \cite{MS23} when $m=1$. {In particular, the implication B. to E. in Theorems \ref{thm:all} and \ref{thm:linear} resolves the question posed in \cite{BLOT24b}*{Remark 3.6}, while at the same time extends the $L^2$-testing theorem to the multilinear setting as well.}

Second, we provide a representation formula for all compact CZOs. This is one answer to the search for examples of compact CZOs. Frankly, from Theorem \ref{thm:all}, we now know all the examples of $m$-linear compact CZOs. This contribution is new even in the linear case. In words, we can describe our wavelet representation formula in terms of two building blocks whose adjoint operators are well-known examples of compact CZOs. The first class of building blocks consists of what we call compact wavelet forms which are almost-orthogonal perturbations of compact approximations of the wavelet resolution of the identity. The second class of building blocks consists of paraproduct forms with $\CMO$ symbol. The wavelet representation theorem we will prove says that all compact CZOs are finite linear combinations of these two building blocks. We refer to Theorem \ref{thm:all} and Definitions \ref{d:wave} and \ref{d:pp} below for precise statements.
}

The strategy of representing CZOs as simpler model operators has been very successful in understanding more refined properties of CZOs, most famously in providing sharp quantitative weighted estimates. Dyadic representation theorems, finding their roots in the works of Figiel \cite{figiel90}; Nazarov, Treil, and Volberg \cite{NTV2}; and Petermichl \cite{PetAJM} among others, have led to a powerful and comprehensive theory of quantitative $L^p$ estimates for singular integrals, culminating with Hyt\"onen's proof of the $A_2$ conjecture \cite{hytonen12}. See \cites{m2022,HL2022,li2019bilinear,martikainen2012representation} and references therein for more recent extensions. Smooth wavelet representations, which may be traced back to David and Journ\'e's original proof of the $T(1)$ theorem, have recently been developed by Di Plinio and some of the authors in a series of papers \cites{DGW2,DGW,DWW} in order to study various Sobolev space estimates for CZOs. 

For simplicity, here in the introduction we state only the bilinear testing characterization, without the representation. It remains only to introduce the ``weak compactness property'' in the bilinear setting.  For the function $\varphi \in L^1_{\mathrm{loc}}(\mathbb R^d)$,  $w \in \mathbb R^d$, and $t >0$, define
	\[\varphi_{z} = \frac{1}{t^d} \varphi\left( \frac{\cdot-w}{t} \right), \quad z=(w,t).\]
Moreover, let $\mathbb B(z)$ denote a hyperbolic ball in the upper half space of radius $1$ centered at $z=(w,t)$ (see \eqref{e:B} {and \eqref{e:limZ}} below for the precise definitions).
Fixing three functions $\phi,\psi^1,\psi^2$, define
	\begin{equation}\label{e:Wintro}\mathsf{W}_{\Lambda}(z) = t^{2d}  \cdot \sup_{ z_1,z_2 \in \mathbb B(z) }\sup_{\substack{\sigma \in \mathrm{S}_3 \\ i=1,2}} \left|\Lambda^\sigma\left(\phi_{z_1}, \psi^i_{z_2} , \phi_{z} \right)\right| ,\end{equation}
where for each $\sigma \in \mathrm{S}_{m+1}$, the permutation group on $\{0,1,\ldots,m\}$, we define
	\begin{equation}\label{e:sigma}\Lambda^\sigma (f_0,\ldots,f_m) = \Lambda(f_{\sigma(0)},\ldots,f_{\sigma(m)}).\end{equation}

 	\begin{theorem}\label{thm:intro}
 	{Let $T$ be a bilinear CZO, $\Lambda$ the associated trilinear form, and $\mathfrak b_j \in \BMO$ defined by \eqref{e:T1}. Then, the following are equivalent.}
	\begin{itemize}
	\item[A.] For each $1<p_1,p_2 \le\infty$, and $\tfrac 12 <r < \infty$ satisfying $\frac 1{p_1} + \frac 1{p_2} = \frac 1r$, $T$ extends to a compact bilinear operator from $L^{p_1}(\R^d) \times L^{p_2}(\R^d)$ to $L^{r}(\R^d)$ and from {$L^\infty(\R^d) \times L^\infty(\R^d)$ to $\CMO$.}
	\item[B.] $\lim\limits_{z \to \infty} \mathsf{W}_\Lambda(z) = 0$ and for each $j=0,1,2$, $\mathfrak b_j \in \CMO$.
\end{itemize}
\end{theorem}

A broad outline of the proof is to use a wavelet averaging procedure to obtain the representation formula in \S \ref{s:rep} and then use the Riesz-Kolmogorov criterion to obtain the precompactness of the image of the unit ball under the adjoint operators to the wavelet and paraproduct forms in \S \ref{s:compact} and \S \ref{s:pp}. {These results are combined in \S \ref{s:all} to state our main result, Theorem \ref{thm:all} including the representation theorem and a simplified statement in the linear case, Theorem \ref{thm:linear}. There we also outline a sample application.}

\section{Wavelets}\label{s:wave} In this section we will review some preliminaries regarding wavelets and introduce one of the building blocks of our representation, wavelet forms.

\subsection{Analysis of the parameter space} ~{For a positive integer $D$,} introduce the parameter space 
	\[ Z^D = \{z = (w,t) : w \in \R^D , \ t>0\},\]
whose elements $z=(w,t)$ act on functions $f\in L^1_{\mathrm{loc}}(\mathbb R^{D})$ by the formula
	\[ \Sy_z f = \frac{1}{t^{D}} f\left( \frac{\cdot-w}{t} \right).\]
Let $\mu$ be the measure on $Z^D$ given by \[
\int\displaylimits_{Z^D}  F(z) \, \d \mu(z) =\int_0^\infty \int_{\R^D} F(w,t) \, \frac{\d w \, \d t}{t}, \quad F \in \C_0(Z^D).
\]
Notice that $\mu$ is invariant under $\Sy_z$. 
To analyze multilinear operators, we will use wavelets adapted to two parameters, one in $Z^{md}$ and the other in $Z^d$. First, given $\mathbf{w} \in \mathbb R^{md}$ and $w_0 \in \mathbb R^d$, define
	\begin{equation}\label{e:vec-norm} |\mathbf w - w_0|_2 = \sqrt{\sum_{i=1}^m |w_i-w_0|^2}, \quad \mathbf{w} = (w_1,\ldots,w_m), \quad w_i \in \mathbb R^d.\end{equation}
For $\mathbf{z} =(\mathbf{w},s)\in Z^{md}$ and $\zeta=(w_0,t) \in Z^d$, define {for any $0 < \delta \le 1$,}
\begin{equation}
	\label{e:Delta1}
	[\mathbf{z},\zeta]_{\delta} = \frac{ \min\{s,t\}^{\delta}}{\max\{t,s,|\mathbf{w}-w_0|_2\}^{md+\delta}}.
\end{equation} We will also say $\mathbf{z} \ge \zeta$ if $s \ge t$. Notice that if  $ \delta \geq \delta'$ then $ [\mathbf{z},\zeta]_{\delta} \leq  [\mathbf{z},\zeta]_{\delta'}.$
For $M \ge 1$ and $\zeta \in Z^d$, introduce $\mathbb B_M^m(\zeta)$ which are the following approximate balls in the hyperbolic metric,
	\begin{equation}\label{e:B} \mathbb{B}_M^m(\zeta) = \left\{ \mathbf{z} \in Z^{md}: t2^{-M} \leq s \leq t2^M, \, \abs{\mathbf{w}-w_0}_2 \leq  t 2^M  \right\} \end{equation}
when $\zeta=(0,1)$, $M=1$, or $m=1$, those parameters are omitted from the notation.
Given a function $F: Z^d \to \mathbb{C}$ we say $\lim\limits_{z \to \infty}F(z)=L$ if
	\begin{equation}\label{e:limZ} \lim_{M \to \infty} \sup_{z \not \in \mathbb{B}_M} |F(z)-L|=0. \end{equation}

\subsection{Wavelet classes and forms}\label{ss:wave}
The building blocks of our representation theorem are wavelets for which we are going to introduce notation and relevant classes as well as the averaging lemmata from \cites{DGW,DWW}.
We denote the space of Schwartz functions by $ \mathcal S(\mathbb R^d)$ and the mean-zero subspace
	\[ \Schw^0(\mathbb R^d)= \left\{ \varphi \in \Schw(\mathbb R^d) : \int \varphi(x) \, \d x =0 \right\}.\] 
We fix a radial function $\phi \in \Schw^0(\mathbb R^d)$, supported in a ball and appropriately normalized which we will call the mother wavelet, 
in which case the Calder\'{o}n reproducing formula holds, namely
\begin{equation}\label{e:calderon} f = \int_{Z^d} \lip f,\phi_z \rip \phi_z\, \d\mu(z) \qquad \forall f\in \mathcal S(\mathbb R^d).  \end{equation} 
 For the convenience of the reader, we restate the setup from \cite{DGW} on which we will base our analysis.
For $ 0 < \delta \leq 1$ we introduce the norm on functions $\varphi \in \C^\delta(\mathbb R^{md})$,
\begin{equation}\label{e:star-norm} \|\varphi\|_{\star,\delta} = \sup_{\substack{x,h \in \R^{md} \\ 0 < |h| \le 1 }} \l x \r^{md+\delta} \left( |\varphi(x)| + \frac{|\varphi(x+h)-\varphi(x)|}{|h|^\delta} \right), \quad \l x \r = \sqrt{1+|x|^2}. \end{equation}
\begin{definition}\label{d:psi}
	For $z=(w,t) \in Z^d$, the wavelet class $\Psi^{m,\delta}_z$ is defined by
	\[ \Psi^{m,\delta}_z = \left\{ \varphi \in \C^\delta (\R^{md}) : \|(\Sy_{\mathbf z})^{-1} \varphi \|_{\star,\delta} \le 1\right\}, \quad \mathbf{z}=(w,\ldots,w,t) \in Z^{md}, \]
	and its cancellative subclass, for $j= 1,\ldots,m$ is denoted by $\Psi^{m,\delta;j}_z$ and consists of $\varphi \in \Psi^{m,\delta}_z$ such that
	\[ \int_{\mathbb R^{d}} \varphi(x_1,x_2,\ldots,x_m) \, \d x_j =0.\]
Let $\chi_z$ denote the $L^{\infty}$-normalized decay factor adapted to the parameter $z=(w,t)$, 
\[ \chi_z(x) = \left \langle \frac{x-w}{t} \right \rangle^{-1}. \]
With this notation, we can recast $\varphi \in \Psi^{m,\delta}_z$ as
	\begin{equation}\label{e:chi} |\varphi| \le \frac{1}{t^{md}} \chi_{\mathbf{z}}^{md+\delta}, \quad |\varphi-\varphi(\cdot+h)| \le \frac{|h|^\delta}{t^{md+\delta}}\chi_{\mathbf{z}}^{md+\delta}.\end{equation}
\end{definition}
Aiming to make this paper as self contained as possible, as well as to set the stage for the representation in Proposition \ref{p:rep}, we state the averaging lemmata from \cite{DGW} and a refinement of the averaging procedure in \cite{DWW}.
\begin{lemma}\label{l:calderon}
	Let $\phi$ be the mother wavelet and $k \ge 0$. There exist functions $\psi^i$, $i=1,2,3,4$, satisfying
	\begin{itemize}
		\item[(i)] $\supp \psi^i \subset B(0,1)$;
		\item[(ii)] $\psi^1,\psi^3 \in \C^k(\mathbb R^d)$;
		\item[(iii)] $\psi^2,\psi^4 \in \Schw^0(\mathbb R^d)$;
		\item[(iv)] For any $s>0$ and $f \in \Schw(\R^d)$,	\[ \int_{r \ge s} \int_{u \in \R^d} \lip f, \phi_{u,r} \rip \phi_{u,r} \dfrac{ \d u \ \d r}{r} = \int_{\R^d} \lip f, \psi_{u,s}^1 \rip \psi_{u,s}^2 + \lip f,\psi_{u,s}^3 \rip \psi_{u,s}^4 \, \d u. \]
	\end{itemize}
\end{lemma}
\begin{lemma}\label{l:avg}
	Let $\varphi_j \in \Schw(\mathbb R^d)$ for $j=1,\ldots,m$ and $0<\eta<\delta \le 1$. There exists $C>0$ such that for any $H:Z^{md} \times Z^d \to \mathbb{C}$ satisfying
	\[\abs{H(\mathbf{z},\zeta)} \le \left [\mathbf{z},\zeta \right ]_{\delta}, \]
	there holds
	\begin{equation}\label{eq:nu} \nu_{\zeta}= \int_{\substack{\mathbf{z} \in Z^{md}\\ \mathbf{z} \ge \zeta}} H(\mathbf{z},\zeta) \left( \varphi_1 \otimes \cdots \otimes \varphi_m\right)_{\mathbf{z}} \d\mu(\mathbf{z}) \in C \Psi^{m,\eta}_{\zeta}. \end{equation}
Furthermore, if $\varphi_j \in \Schw^0(\mathbb R^d)$, then $\nu_{\zeta} \in C\Psi^{m,\eta;j}_{\zeta}$.
\end{lemma}

Finally, we end this section by introducing the wavelet forms, which will be used to systematically study $m$-linear CZOs.
\begin{definition}\label{d:wave}
	Given a collection $\{\ep_z \in \mathbb C , \nu_z \in \Psi^{m,\delta;1}_z : z \in Z^d\}$ define the associated $(m+1)$-linear \textit{canonical wavelet form} by
	\[ \int_{Z^d} \eps_z \ip{f_0}{\phi_z} \ip{f_1 \otimes \ldots \otimes f_m}{\nu_z} \, \d \mu(z).\]
More generally, we say $U$ is a wavelet form if $U^\sigma$ is a canonical wavelet form for some $\sigma \in \mathrm{S}_{m+1}$. We say $U$ is a \textit{bounded wavelet form} if $\sup_{z \in Z^d} \abs{\ep_z} < \infty$. Additionally, we say $U$ is a \textit{compact wavelet form} if it is a bounded wavelet form for which
	\[ \lim_{z \to \infty} \ep_z =0.\]
\end{definition}
Wavelet forms may be viewed as a generalization of the Calder\'on-Toeplitz operators considered in \cites{nowak1993calderon,rochberg1990toeplitz}. More generally though, all cancellative CZ forms {(see Definition \ref{d:cz} below)} can be realized as wavelet forms \cites{DGW,DWW}. {Let us proceed to describe in what sense bounded wavelet forms are indeed bounded.} They themselves are cancellative CZ forms so Proposition \ref{p:wave-bound} below follows from any number of results \cites{GT02,LN19,li2021endpoint,li2014sharp,nieraeth2019quantitative}. See also \cite{DGW}*{Proposition 5.1} for a direct proof of the sparse $(1,\ldots,1)$ bound for bounded wavelet forms. Let us introduce some bookkeeping to concisely describe the full range of Lebesgue space estimates for CZOs.
\begin{definition}\label{d:Holder}
Let 
	\[ P = \left\{ (p_1,\ldots,p_m) : 1<p_j \le \infty\right\}, \quad {P_\circ = \left\{ (p_1,\ldots,p_m) : 1<p_j < \infty\right\} }.\]
We introduce the shorthand for $\vec p \in P$,
	\[ L^{\vec p}(\mathbb R^d) =  \bigtimes_{j=1}^m L^{p_j}(\mathbb R^d) , \quad 
	\mathcal B^{\vec p} = \left\{ (f_1,\ldots,f_m) \in L^{\vec p} : \norm{f_j}_{L^{p_j}(\R^d)} \le 1\right\}, \quad
	r(\vec p) = \left( \sum_{j=1}^m \frac{1}{p_j}\right)^{-1}.\]
Then the admissible classes of H\"older tuples we consider are
	\[ Q = \left\{ (\vec p, r(\vec p)) : \vec p \in P, \ r(\vec p)<\infty \right\}, \quad {Q_{\mathrm{ref}} = \left\{ (\vec p,r(\vec p) : \vec p \in P_{\circ}, \ r(\vec p) > 1\right\} }.\]
Given $(\vec p,r) \in Q$, and an $m$-linear operator $T$, denote by $\norm{T}_{\vec p,r}$ the operator norm from $L^{\vec p}(\mathbb R^d) \to L^r(\mathbb R^d)$, i.e. 
	\[ \norm{T}_{\vec p,r} = \sup_{(f_1,\ldots,f_m) \in \mathcal B^{\vec p}} \norm{ T(f_1,\ldots,f_m)}_{L^r(\R^d)}.\]
Furthermore, we define the following modification at the endpoint,
	\[ \norm{T}_{\infty,\mathrm{BMO}} = \sup_{(f_1,\ldots,f_m) \in \mathcal B^{\vec p}} \norm{ T(f_1,\ldots,f_m)}_{\BMO}, \quad \vec p =(\infty,\ldots,\infty),\]
{where the $\BMO$ norm is defined by
	\begin{equation}\label{e:BMO-norm} \norm{f}_{\BMO} = \sup_{Q \, \mathrm{cube}} \frac{1}{|Q|} \int_Q \abs{ f(x)-f_Q } \, \d x, \quad f_Q = \frac{1}{|Q|} \int_Q f(y) \, \d y.\end{equation}
}\end{definition}

\begin{proposition}\label{p:wave-bound}
For each $(\vec p,r) \in Q$, there exists $C_{\vec p,r}>0$ such that for any $T$ such that $U(f_0,\ldots,f_m) = \ip{f_0}{T(f_1,\ldots,f_m)}$ is a bounded wavelet form,
	\begin{equation}\label{e:wave-bound}\norm{T}_{\vec p,r} \le C_{\vec p,r} \sup_{z \in Z^d} \abs{\ep_z}. \end{equation}
Furthermore, there exist $C_\infty>0$ {and $\bar T$ an extension of \,$T$} such that
	\begin{equation}\label{e:wave-bound-endpoint} \norm{ \bar T}_{\infty,\mathrm{BMO}} \le C_\infty \sup_{z \in Z^d} \abs{\eps_z}.\end{equation}
\end{proposition}

\section{Wavelet representation of compact Calder\'on-Zygmund forms}\label{s:rep}
Let us make rigorous the informal definitions given in the introduction. 
\begin{definition}\label{d:si}
A function $K\in L^1_{\mathrm{loc}}( \mathbb R^{(m+1)d} \setminus \{x \in (\mathbb R^d)^{m+1} : x_0=\cdots=x_m\})$ is a $\delta$-\textit{singular integral kernel} if there exist $C_K,\delta>0$ such that
	\begin{align}\label{e:ker-size} &|K(x_0,\ldots,x_m)| \le \dfrac{C_K}{\abs{x_0-\mathbf{x}}_2^{md}}, \quad \mathbf x = (x_1,\ldots,x_m),\\
	\label{e:ker-smooth}& \max_{j=0,\ldots,m} |\Delta^j_h K(x_0,\ldots,x_m)| \le \dfrac{C_K|h|^\delta}{\abs{x_0-\mathbf{x}}_2^{md+\delta}},\end{align}
where $\Delta_h^j$ denotes the difference operator in the $j$-th position{, and $\abs{\cdot}_2$ is defined in \eqref{e:vec-norm}}.
{Let $\Lambda: \times_{j=0}^m \Schw(\mathbb R^d) \to \mathbb C$ be a continuous $(m+1)$-linear form. We say $\Lambda$ is an $(m+1)$-linear \textit{singular integral form} }if there exists a singular integral kernel $K$ such that for any $f_0,\ldots,f_m \in \Schw(\mathbb R^d)$ with $\cap_{j=0}^m \supp f_j= \varnothing$ one has
	\[ \Lambda(f_0,\ldots,f_m) = \int_{(\mathbb R^d)^{m+1}} K(x) \prod_{j=0}^m f_j(x_j) \, \d x. \]
When $m$ is understood, we simply say $\Lambda$ is a singular integral form.
\end{definition}
 
Let $\C_{\mathrm{v}}(\mathbb R^d)$ be the space of all continuous functions $f$ on $\R^d$ for which $\lim_{|x|\to\infty} f(x) =0$. Then, the Banach space of functions with continuous mean oscillation, $\CMO$, is defined to be the closure of $\C_{\mathrm{v}}(\mathbb R^d)$ in the norm $\norm{\cdot}_{\BMO}${, defined in \eqref{e:BMO-norm} above}. 

\begin{definition}\label{d:cz}
We say a singular integral form $\Lambda$ is a \textit{Calder\'on-Zygmund (CZ) form} if there exists $C_{\mathsf{W}}$ such that
	\[  t^{md}\abs{ \Lambda( \Sy_z \varphi_0,\ldots,\Sy_z\varphi_m) } \le C_{\mathsf{W}}, \quad \forall \varphi_j \in \C^\infty_0(B(0,1)) \cap \Psi^{1,1}_{(0,1)},\]
and furthermore, there exist $\bb_j \in \BMO$ such that \eqref{e:T1} holds. The rigorous definition of \eqref{e:T1} is as follows. Let $\theta \in \C^\infty_0(\mathbb R^d)$ with {$\theta(x)=1$ for $|x| \le 1$}. Then, for $t>0$, set $\theta_t = \theta(\cdot t)$. We say \eqref{e:T1} holds if for all $\varphi \in \Schw^0(\mathbb R^d)$,
	\begin{equation}\label{e:T1rig}\lim_{t\to 0} \Lambda^{*,j}(\varphi,\theta_t,\theta_t,\ldots,\theta_t) = \ip{\varphi}{\bb_j}.\end{equation}
{If $T$ is a continuous $m$-linear operator from $\times_{j=1}^m \Schw \to \Schw'$ such that $\ip{f_0}{T(f_1,\ldots,f_m)} = \Lambda(f_0,\ldots,f_m)$ for some CZ form $\Lambda$, then $T$ is said to be a Calder\'on-Zygmund operator (CZO).}
Furthermore, let $\phi$ be a mother wavelet and $\psi^2,\psi^4$ be the Schwartz functions from Lemma \ref{l:calderon}. For $j=1,\ldots,m$ define
	\begin{equation}\label{e:psi-e} \vec \psi_j = \left(\bigotimes_{i=1}^{j-1} \psi^2\right) \otimes \left( \bigotimes_{i=j}^{m-1} \psi^4 \right) \otimes \phi.\end{equation}
For each $M \ge 1$, define 
	\begin{equation}\label{e:WM}\mathsf{W}_\Lambda^M(\zeta) = \sup_{\mathbf z \in \mathbb B^m_M(\zeta)} \sup_{\substack{\sigma \in \mathrm{S}_{m+1} \\ j=1,\ldots,m }} \abs{ \Lambda^\sigma\left( (\vec \psi_j)_{\mathbf{z}},\phi_\zeta \right) } t^{md}, \quad \zeta = (w,t).\end{equation}
A CZ form $\Lambda$ is said to be a compact CZ form if for some (all) $M \ge 1$
	\[ \bb_j \in\CMO, \quad \lim\limits_{\zeta \to\infty} \mathsf{W}_\Lambda^M(\zeta)=0.\] 
Finally, we say a CZ form is cancellative if $\bb_j=0$.
\end{definition}
\begin{remark}\label{r:cz}
{By the symmetry in Definitions \ref{d:si} and \ref{d:cz} above, if $\Lambda$ is a CZ form or compact CZ form, then so is $\Lambda^\sigma$ for any $\sigma \in \mathrm{S}_{m+1}$, defined by \eqref{e:sigma}. Therefore, if $T$ is a CZO, so is $T^{*,j}$ for any $j=1,\ldots,m$.}
It is not too difficult to see that if $\mathsf{W}_\Lambda^M \to 0$ for some $M$ then the same holds for each $M$. {Therefore, when convenient, we may reduce to the case $M=1$ and utilize the shorthand $\mathsf{W}_\Lambda(\zeta)=\mathsf{W}_\Lambda^1(\zeta)$.}
\end{remark}
{   
\begin{remark}\label{r:ex}
	Compact wavelet forms, as defined in Definition \ref{d:wave} are compact CZ forms. Almost orthogonality of elements of the wavelet classes, e.g. \cite{DWW}*{Lemma 2.3}, and the fact that $\ep_z \to 0$ readily shows $\mathsf{W}_\Lambda(z) \to 0$ for such forms. Furthermore, since wavelet forms possess cancellation in \textit{two} inputs, all the symbols $\mathfrak b_j$ are zero.
	Compact paraproduct forms, the subject of \S \ref{s:pp} below, are also compact CZ forms. For these, $\mathsf{W}_\Lambda(z) \to0$ for the same reason as the wavelet forms. Furthermore, $\mathfrak b_0$ is the paraproduct symbol while $\mathfrak b_j=0$ for $j=1,\ldots,m$.
	Our representation theorem (Theorem \ref{thm:all} part C. below) states that in fact all compact CZ forms are a linear combination of forms from these two classes.
\end{remark}
}
Now we justify our description of such forms as compact.

\begin{definition}\label{d:compact} Let $T$ be an $m$-linear CZO. We say $T$ is a compact CZO if for each $(\vec p ,r) \in Q$, {there exists an extension of $T$, $\bar T$, such that} $\bar T(\mathcal B^{\vec p})$ is precompact in $L^r(\R^d)$, and at the upper endpoint, there exists an extension $\bar{\bar T}$ such that $\bar{\bar T}(\mathcal B^{\infty})$ is precompact in $\CMO$. {Furthermore, if the weaker condition, for all $(p,r) \in Q_{\mathrm{ref}}$ $T(\mathcal B^{{\vec p}})$ is precompact in $L^r(\mathbb R^d)$, holds, then we say $T$ is a compact CZO in the reflexive range.}
\end{definition}

\begin{remark}\label{r:compact}
The classical definition of compactness of an abstract $m$-linear operator on quasi-normed spaces \cites{BT13,COY} is that it maps bounded sets to precompact sets. In Definition \ref{d:compact}, we are imposing this definition of compactness on $T$ acting from $L^{\vec p}(\R^d) \to L^r(\R^d)$ for \textit{all} $(\vec p ,r)$ in $Q$ and at the upper endpoint because our testing conditions allow us to conclude compactness in this full range. It is tempting to only require $T:L^{\vec p}(\R^d) \to L^r(\R^d)$ be compact for a single $(\vec p,r)$, but the current state of multilinear extrapolation of compactness \cites{COY,HL2} does not include the endpoints. 
\end{remark}

{    
An important concept closely related to compact operators is that of a weakly convergent sequence. We give a straightforward sufficient condition for generating such families in $L^r(\mathbb R^d)$ parameterized by $z \in Z^d$.
\begin{lemma}\label{l:weakly}
Let $1\le q_-< r < q_+\le\infty$ and $\{ f_z \in L^{q_-}(\mathbb R^d) \cap L^{q_+}(\mathbb R^d) : z \in Z^d\}$ {be a collection of functions satisfying}
	\begin{itemize}
	\item[(i)] $\norm{f_z}_{L^r}$ is uniformly bounded,
	\item[(ii)] there exists $\beta>0$ such that $\norm{f_z}_{L^{q_-}} \lesssim t^\beta$, $\norm{f_z}_{L^{q_+}} \lesssim t^{-\beta}$,
	\item[(iii)] {There exists $c>0$ such that for $|x-w| \ge c \cdot t$}, $\abs{f_z(x)} \lesssim t^{-\frac dr} \chi_z(x)$.
	\end{itemize}
Then for any $g \in L^{r'}(\mathbb R^d)$, $\lim\limits_{z \to \infty}\ip{f_z}{g} \to 0$.
\end{lemma}
\begin{proof}
By (i) and a standard density argument, it is enough to prove the conclusion for $g \in \C^\infty_0(\mathbb R^d)$. Suppose $\supp g \subset B(0,R)$. The main step is the trivial consequence of H\"older's inequality.
	\[ \abs{ \ip{f_z}{g} } \le \min \left\{ \norm{f_z}_{L^\infty(B(0,R))} \norm{g}_{L^1}, \norm{f_z}_{L^{q_-}} \norm{g}_{L^{q_-'}},\norm{f_z}_{L^{q_+}} \norm{g}_{L^{q_+'}} \right\}.\]
(iii) ensures that the first quantity is smaller when $z=(w,t)$, $t \sim 1$, and $|w|$ is much larger than $R{+c}$. (ii) controls the latter two quantities when $t$ is small or large.
\end{proof}
\begin{proposition}\label{p:nec}
Let $T$ be an $m$-linear CZO. First, if $T$ is compact in the reflexive range, then for any $\varphi^1,\ldots,\varphi^m \in L^\infty$ with compact support, and any $j=0,1,\ldots,m$,
	\begin{equation}\label{e:local-test} \lim_{z \to \infty} t^{md-\frac d2}\norm{T^{*,j} (\varphi ^1_z,\ldots,\varphi ^m_z)}_{L^2(\mathbb R^d)}=0, \quad \varphi_z^i = \Sy_z \varphi^i.\end{equation}
Second, if \eqref{e:local-test} holds, then each $\bb_j$ belongs to $\CMO$ for $j=0,\ldots,m$, and for each $M \ge 1$,
	\begin{equation}\label{e:W} \lim_{\zeta \to \infty} \mathsf{W}_\Lambda^M(\zeta) =0.\end{equation}
\end{proposition}
\begin{proof} To prove the first conclusion, set $b_z^j = t^{md-\frac d2} T^{*,j} (\varphi ^1_z,\ldots,\varphi ^m_z)$. Since $T^{*,j}$ is bounded from $\times_{j=1}^m L^{qm}(\mathbb R^d) \to L^q(\mathbb R^d)$ for any $1<q<\infty$ (being a CZO) we have 
	\begin{equation}\label{e:b-norm-q} \norm{b_z^j}_{L^q} \lesssim t^{\frac dq-\frac d2}. \end{equation} 
Furthermore,
	\begin{equation}\label{e:id} \norm{b_z^j}_{L^2}^2 = t^{md-\frac d2} \ip{b_z^j}{T^{*,j} (\varphi ^1_z,\ldots,\varphi ^m_z)} = \ip{f^0_z}{T(f^1_z,\ldots,f^m_z)},\end{equation}
for $f^i_z$ which are appropriate permutations and normalizations of $\varphi ^i_z$ and $b^j_z$. To assign them precisely, we distinguish between $j=0$ and $j =1,\ldots,m$. In the first case, we can simply take
	\[ f_z^i = 
		\left\{ \begin{array}{ll}
			t^{d(1-\frac 1{2m})}\varphi _z^i, & i =1,\ldots,m; \\
			b_z^0, & i=0;\end{array} \right.  \quad p_i = 2m, \ i=1,\ldots,m; \quad r=2.\]
And for $j =1,\ldots,m$ then
	\[ f_z^i = 
		\left\{ \begin{array}{ll}
			t^{d(1-\frac{1}{4(m-1)})} \varphi _z^i, & i \not\in\{0,j\}; \\
			t^{d(1-\frac1{4})}\varphi _z^{j}, & i=0; \\
			b_z^j, & i=j; \end{array} \right.  \quad 
	p_i = 
		\left\{ \begin{array}{ll}
			 4(m-1), & i \not\in\{0,j\}; \\
			2, & i=j; \end{array} \right. \quad r=4.  \]
In either case, as defined, 
$\{ (f^1_z,\ldots,f^m_z), z \in Z^d \}$ is a bounded set in $L^{\vec p}$. Crucially, by Lemma \ref{l:weakly}, $f^0_z$ converges weakly to zero in $L^{r'}(\mathbb R^d)$. In the case $f^0_z=t^{d(1-\frac 1r)}\varphi ^{j}_z$ the conditions are readily verified for $q_-=1$ and $q_+=\infty$. The more challenging case, when $f^0_z = b_z^0$ (the case $j=0$ and $r=2$), also satisfies the assumptions of Lemma \ref{l:weakly}. Indeed, (i) and (ii) are consequences of \eqref{e:b-norm-q}. {To establish (iii), take $c>0$ large enough that all $\varphi^j$ are supported in the ball $B(0,\frac c2)$. Then, since $\varphi^j_z$ is supported in the ball $B(w,\frac {ct}2)$, for $x$ outside $B(w,ct)$, \eqref{e:ker-size} implies $K(x,\mathbf x) \lesssim |x-w|^{-md}$. The condition (iii) is then a consequence of
	\[ b^0_z(x) 
	= t^{md-\frac d2}\int_{\mathbb R^{md}} K(x,\mathbf x) \prod_{j=1}^m \varphi_z^j(x_j)  \, \d\mathbf x \lesssim t^{-\frac d2} \left( \frac{|x-w|}{t} \right)^{-md} \lesssim t^{-\frac d2} \chi_z(x). \]
}Now suppose \eqref{e:local-test} fails. Then, by \eqref{e:id}, there exist $\ep>0$ and $z_n \to \infty$ such that
	\begin{equation}\label{e:contra} \abs {\ip{f^0_{z_n}}{T(f^1_{z_n},\ldots,f^m_{z_n})} } \ge \ep.\end{equation}
As $T$ was assumed to be compact from $L^{\vec p} \to L^r$, there exists a subsequence $\zeta_n \to \infty$ such that $\{T(f^1_{\zeta_n},\ldots,f^m_{\zeta_n})\}$ converges in $L^r$. Combined with the weak convergence to zero of $f^0_{\zeta_n}$, this contradicts \eqref{e:contra}.

To prove the second statement, Cauchy-Schwarz immediately tells us that \eqref{e:W} follows from \eqref{e:local-test}. The proof of $\mathfrak b_j \in\CMO$ is more delicate. To do so, we define the function $\mathfrak c_j \in \BMO$ by $\mathfrak c_j = T^{*,j}[1,1,\ldots,1]$, where $T[\cdot]$ denotes the now classical extension of $T$ to elements of $\times_{j=1}^m L^\infty(\mathbb R^d)$, 
	\begin{equation}\label{e:Linfty}T[f_1,\ldots,f_m]:= \lim_{t\to 0} T(f_1 \theta_t,\ldots,f_m \theta_t) - c_t, \end{equation}
{where the limit exists, distributionally, locally in $L^1$, and pointwise a.e.}, and $c_t$ are suitable constants depending $(m+1)$-linearly on $T,f_1,\ldots,f_m$ with $\lim_{t \to 0} c_t = c_0 \in \mathbb C$; see e.g. \cite{GT02}*{Lemma 1 and Proposition 1}. Notice that when $f_j \in L^\infty(\mathbb R^d) \cap \mathcal S(\mathbb R^d)$, the definition \eqref{e:Linfty} differs from $T(f_1,\ldots,f_m)$ but only by a constant. Therefore, when we view \eqref{e:Linfty} as an element of $\BMO$, which consists of equivalence classes modulo constants, there is no contradiction. Clearly $\mathfrak c_j =\mathfrak b_j$ in our distributional sense \eqref{e:T1rig} as well. It is therefore enough to show $\mathfrak c_j \in \CMO$, which is well-known \cite{Uch78} to be equivalent to showing for each $\ep>0$ there exists $M>0$ such that
	\[ \sup_{z \not\in \mathbb B_M} \osc_z(\mathfrak c_j) < \ep, \quad \osc_z(f) = \frac{1}{t^d} \int_{B(w,t)} \abs{f-f_z} , \quad f_z= \frac{1}{t^d} \int_{B(w,t)} f .\] Let $\varphi \in L^\infty$ satisfy $1_{\{|x| \le 1 \}} \le \varphi \le 1_{\{|x| \le 2 \}}$ and for $\zeta = (\xi,\sigma) \in Z^d$ to be chosen later, we {set $\varphi_\zeta = \sigma^d \Sy_\zeta \varphi$}. In the sense of \eqref{e:Linfty}, split $\mathfrak c_j = E + T^{*,j}[\varphi _\zeta,\ldots,\varphi _\zeta]$, where
	\[ E = T^{*,j}[1-\varphi _\zeta,1,\ldots,1] + T^{*,j}[\varphi _{\zeta},1-\varphi _{\zeta},1,\ldots,1] + \cdots + T^{*,j}[\varphi _{\zeta},\ldots,\varphi _{\zeta},1-\varphi _{\zeta}].\]
{Let now $z=(w,t)$ and we will estimate $\osc_z(E)$ using the smoothness estimate for the kernel \eqref{e:ker-smooth}. To this end, for $s>0$, let $E_s^k = T^{*,j}(\varphi_\zeta \theta_s,\ldots,\varphi_\zeta\theta_s,(1-\varphi_\zeta)\theta_s,\theta_s,\ldots,\theta_s)$, where $(1-\varphi_\zeta)$ is in the $k$-th position. Then, for any $a \ge 2$, set $\zeta = (w,at)$. The following estimate holds uniformly over $s>0$, $a \ge 2$, and $x,y \in B(w,t)$,
	\[ |E_s^k(x) - E_s^k(y)| \lesssim t^\delta \int_{\substack{\mathbf x \in \mathbb R^{md} \\ |x_{k}-x| \gtrsim at}} |x-\mathbf x|_2^{-md-\delta} \, \d \mathbf x.\]
The remaining integral is controlled by $(at)^{-\delta}$. Passing to the limit, $|E(x)-E(y)| \lesssim a^{-\delta}$ and hence $\osc_z(E) \lesssim a^{-\delta}$.} Second, by H\"older's inequality and the fact that $T$ and $T[\cdot]$ only differ by a constant,
	\[ \osc_z(T^{*,j}[\varphi _\zeta,\ldots,\varphi _\zeta]) = (at)^{\frac d2} \osc_z(b^j_\zeta) \lesssim_a \norm{b^j_{\zeta}}_{L^2}.\] 
Therefore, given $\ep>0$, first pick $a$ large enough that $\osc_z(E) < \frac \ep 2$. Then, by \eqref{e:local-test}, we can pick $M$ large enough that for $\zeta \not\in \mathbb B_M$, $\osc_z(T^{*,j}[\varphi _\zeta,\ldots,\varphi _\zeta]) < \frac \ep 2$. But there exists $N \gtrsim_a M$ such that if $z =(w,t) \not\in \mathbb B_N$ then $\zeta=(w,at) \not \in \mathbb B_M$.
\end{proof}
}

The main step in the representation theorem is given in Proposition \ref{coeffdecaycondition} below. There we will show that a compact cancellative CZ form enjoys additional decay in the wavelet basis. For a general cancellative CZ form, we recall the following lemma from \cite{DGW}*{Lemma 3.3} regarding its decay when applied to a $(m+1)$-tuple of wavelets. When we want to specify the value of the smoothness parameter $\delta>0$ from Definition \ref{d:si}, we say $\Lambda$ is a $\delta$-CZ form.
\begin{lemma}\label{l:cz} Let $\Lambda$ be a cancellative $\delta$-CZ form, $\eta \in (0,\delta)$, and $\psi_j \in \C^{\infty}_0( B(0,1))$, $j=1,\ldots,m-1$. Then, there exists $C>0$ such that for all $\sigma \in \mathrm{S}_{m+1}$, $\zeta \in Z^d$, and $\mathbf{z} \in Z^{md}$ with $\mathbf{z} \ge \zeta$,
	\[	\left|\Lambda^\sigma \left( \vec\psi_{\mathbf{z}},\phi_{\zeta} \right ) \right| \le C \left [\mathbf{z},\zeta \right ]_{\eta}, \quad \vec \psi = \psi_1 \otimes \cdots \otimes \psi_{m-1} \otimes \phi.\]
\end{lemma}

\begin{proposition} \label{coeffdecaycondition}
Let  $\Lambda$ be a cancellative compact $\delta$-CZ form and $\eta \in (0,\delta)$. Let $\phi$ be the mother wavelet and let $\vec \psi_j$ be defined as in \eqref{e:psi-e}. Then there exists $C>0$ and $\left\{\eps_\zeta \in \mathbb{C} : \zeta \in Z^d \right\}$ with the properties that \begin{equation}\label{e:eps}  \sup_{\zeta \in Z^d} \abs{\eps_\zeta} \le C, \quad \lim_{\zeta \to \infty} \abs{\eps_\zeta}=0, \quad \max_{\substack{\sigma \in \mathrm{S}_{m+1} \\ j=1,\ldots,m }} \left|\Lambda^\sigma  \left( (\vec\psi_j)_{\mathbf{z}}, \phi_{\zeta}\right ) \right| \leq \abs{\eps_{\zeta}} \left[ \mathbf{z},\zeta\right]_{\eta}, \quad \mathbf{z} \geq \zeta.    \end{equation}
\end{proposition}

\begin{proof}
By symmetry, we can assume $\sigma$ is the identity element and $j=1$. Furthermore, set $\vec \psi=\vec \psi_1$ to declutter the notation. For each $ n \in \N$, we will construct $ \left\{\eps_{\zeta}^{n} \in \mathbb C : \zeta \in Z^d\right\} $ and $\rho>0$ with the properties that\[ \begin{aligned}
	& \sup_{\zeta \in Z^d} |\eps_{\zeta}^n| \lesssim 1, \quad \lim_{\zeta \to \infty} |\eps_\zeta^n|=0, \\ 
	&\left|\Lambda  \left( \vec\psi_{\mathbf{z}}, \phi_{\zeta}\right ) \right| \lesssim 2^{-\rho n} \left|\eps^{n}_{\zeta}\right| \left[ \mathbf{z},\zeta \right]_{\eta},  \quad \mathbf{z} \in \left\{ \begin{array}{ll} \mathbb{B}_{1}^m(\zeta)  & n=1, \\ \mathbb{B}_{n}^m(\zeta) \setminus \mathbb{B}_{{n-1}}^m(\zeta) & n \ge 2,\end{array}  \right.  \quad \mathbf{z} \ge \zeta. \end{aligned} 
	   \]
Assuming we have such $\ep_\zeta^n$, define
\[ \eps_{\zeta} = \sum_{n =1}^\infty 2^{-\rho n} \eps^n_{\zeta}.   \]
The first and third properties of $\eps_\zeta$ in \eqref{e:eps} are immediate and the second follows by Lebesgue's dominated convergence theorem since each $\ep^n_\zeta$ approaches zero. Now, to construct $\ep_\zeta^n$, let $\delta_j >0$ such that $\eta<\delta_1 <\delta_2<\delta$. By Lemma \ref{l:cz} and the definition of $\mathsf{W}_\Lambda^M$ from \eqref{e:WM}, we have for any $\theta \in (0,1)$,
	 \[ \left|\Lambda \left( \vec\psi_{\mathbf{z}},\phi_{\zeta} \right) \right| \lesssim \left(t^{-md} \right)^{1-\theta} \left [\mathbf{z},\zeta \right ]_{\delta_2}^{\theta} \min\{1, \mathsf{W}_\Lambda^{n}(\zeta)\}^{1-\theta},\quad \mathbf{z} \in \mathbb B^m_{n}(\zeta), \ \mathbf{z} \ge \zeta.\]
We choose $\theta$ close enough to $1$ that $md \theta+\delta_2 \theta -md=\delta_1$.
With this specific choice of $\theta$, one can easily verify that $ \left( t^{-md} \right)^{1-\theta}[\mathbf{z},\zeta ]_{\delta_2}^{\theta}= [\mathbf{z},\zeta ]_{\delta_1}$. Setting $\eps_\zeta^1 = \min\{1,\mathsf{W}^1_\Lambda(\zeta)\}^{1-\theta}$ handles the case $n=1$. For $n \ge 2$, we factor, with $\rho=\delta_1-\eta$,
\[ 	\left [\mathbf{z},\zeta \right ]_{\delta_1}= \left(\frac{t}{\max\left\{s,|\mathbf{w}-w_0| \right\} } \right)^{\rho}  \left [\mathbf{z},\zeta \right ]_{\eta}.\]
Therefore, one only needs to verify that for $\mathbf{z} \in \mathbb B^m_{n}(\zeta) \setminus \mathbb B^m_{{n-1}}(\zeta)$, the first factor is comparable to $2^{-\rho n}$. The proof is concluded by setting $\ep^n_\zeta = \min\{1,\mathsf{W}^n_\Lambda(\zeta)\}^{1-\theta}$.
\end{proof}

\begin{proposition}\label{p:rep} Every compact cancellative CZ form is a finite sum of compact wavelet forms.
\end{proposition}
\begin{proof}
	We recycle the proof of the representation theorem in \cite{DGW} relying on Lemma \ref{l:calderon} and Lemma \ref{l:avg}. {Let $f_j \in \Schw(\mathbb R^d)$, $j=0,1,\ldots,m$ and expand each $f_j$ using \eqref{e:calderon} to obtain}
	\[ \Lambda(f_0,\ldots,f_m) = \int_{(Z^d)^{m+1}} \Lambda(\phi_{z_0}, \ldots,\phi_{z_m} ) \ip{f_0}{\phi_{z_0}} \, \d\mu(z_0) \ldots \ip{f_m}{\phi_{z_m}} \, \d\mu(z_m).\]
Split the integration region into $m(m+1)$ components  defined by
	\[ \{ (z_0,\ldots,z_m) \in (Z^d)^{m+1} : z_i \ge z_j \ge z_k, \, i \ne j,k\}, \quad  k=0,\ldots,m, \, j=0,\ldots,k-1,k+1,\ldots,m.\]
For each $j,k$ and  each $i \ne j,k$, we apply Lemma \ref{l:calderon} to $\ip{f_i}{\phi_{z_i}} \phi_{z_i}$. Furthermore, setting $f^\sigma = \bigotimes_{j=0}^m f_{\sigma(j)}$ and relabelling the variables, we obtain
	\[ \Lambda(f_0,\ldots,f_m) = \sum \int_{\zeta \in Z^d} \int_{\substack{\mathbf{z} \in Z^{md} \\ \mathbf{z} \ge \zeta}} \Lambda^\sigma( (\vec \psi_{\mathrm{e}})_{\mathbf{z}},\phi_\zeta) \ip{f^\sigma}{(\vec\psi_{\mathrm{o}})_{\mathbf{z}} \otimes \phi_\zeta } \d \mu(\mathbf{z}) \, \d\mu(\zeta), \]
where the sum is taken over all $\sigma \in \mathrm{S}_{m+1}$ and over the combinations $\vec\psi_{\mathrm{e}}$ of the form \eqref{e:psi-e} for some $j=1,\ldots,m$ and $\vec\psi_{\mathrm o}$ is of the same form but with $\psi^2$ and $\psi^4$ replaced by $\psi^1$ and $\psi^3$ from Lemma \ref{l:calderon}. Each of these summands (of which there are only finitely many depending on $m$), will now be converted to a compact wavelet form by Proposition \ref{coeffdecaycondition} and Lemma \ref{l:avg}. Fix now one $\sigma$, $\vec\psi_{\mathrm e}$, and $\vec\psi_{\mathrm{o}}$. We define
	\[ \varu_{\zeta} =  \int_{\substack{\mathbf{z} \in Z^{md} \\ \mathbf{z} \ge \zeta}} \Lambda^\sigma( (\vec \psi_{\mathrm{e}})_{\mathbf{z}},\phi_\zeta)  ( \vec\psi_{\mathrm{o}} )_{\mathbf z} \, d\mu(\mathbf{z})\]
and the result will be proved if we can show $\varu_\zeta = \eps_\zeta \nu_\zeta$ for some $\nu_\zeta \in \Psi^{m,\delta;m}_\zeta$ and $\ep_\zeta$ approaching zero as $\zeta \to \infty$. Let $\ep_\zeta$ be the collection provided by Proposition \ref{coeffdecaycondition}. In particular, the third property in \eqref{e:eps} guarantees
	\[ \Lambda^\sigma( (\vec \psi_{\mathrm{e}})_{\mathbf{z}},\phi_\zeta) = \eps_\zeta H(\mathbf{z},\zeta), \quad \abs{H(\mathbf{z},\zeta)} \le [\mathbf{z},\zeta]_\eta,\]
whence the proof is concluded by Lemma \ref{l:avg} and recalling that the $m$-th component of $\vec\psi_{\mathrm{o}}$ is the mother wavelet and thus belongs to $\Schw^0(\mathbb R^d)$.
\end{proof}

\section{Compact wavelet forms} \label{s:compact}
In this section we will prove the compact analogue of Proposition \ref{p:wave-bound}. 
\begin{proposition}\label{p:wave-compact} Let $T$ be such that $U(f_0,\ldots,f_m) = \ip{f_0}{T(f_1,\ldots,f_m)}$ is a compact wavelet form. Then $T$ is a compact CZO.
\end{proposition}
\begin{proof} $T(f_1,\ldots,f_m)(x)$ is either of the form
	\begin{equation}\label{e:Topt}  \int_{Z^d}\eps_z \ip{f_1 \otimes \cdots \otimes  f_m}{\nu_z} \phi_z(x) \, \d\mu(z) \quad \mathrm{or} \, \int_{Z^d} \eps_z \ip{f_{\sigma(1)} \otimes \cdots \otimes f_{\sigma(m)} }{\nu_z(x,\cdot) \otimes \phi_z}\, \d \mu(z)\end{equation}
for some $\sigma \in \mathrm{S}_{m}$ and $\nu_z \in \Psi^{m,\delta;j}_z$; by $\nu_z(x,\cdot)$ we mean for each $x$ it returns the function $(x_2,\ldots,x_m) \mapsto \nu_z(x,x_2,\ldots,x_m)$. We will only handle the second case, and we will reduce, by symmetry, to the case where $\sigma$ is the identity. The first case in \eqref{e:Topt} is simpler, though in fact they are handled in exactly the same fashion. Let $\rho>0$ and split the integral defining $T$ over $\mathbb B_M$ and $Z^d \setminus \mathbb B_M$ where $M$ is chosen large enough that $|\ep_z| \le \rho$ for $z \not\in \mathbb B_M$. Therefore, the operator norm of the second component is, by \eqref{e:wave-bound} controlled by $\rho$. The proof will be concluded if we can show $R_\rho$ defined by
	\[ R_\rho \mathbf f (x) = \int_{\mathbb B_M} \ip{f_{1} \otimes \cdots \otimes f_{m} }{\nu_z(x,\cdot) \otimes \phi_z} \, \d\mu(z), \quad \mathbf {f} = (f_1,\ldots,f_m),\]
is compact. By the Riesz-Kolmogorov compactness criteria (see e.g. \cite{MSWW} and \cite{Tsuji} for the case $0<r<1$), we need to prove that 
\begin{align}
	\label{e:RK1}
		& \lim_{N \to \infty}\sup_{\mathbf{f} \in \mathcal{B}^{\vec p}} \int_{|x|>N} |R_{\rho} \mathbf{f}(x)|^r \d x=0, \\ 
	\label{e:RK2}
		& \lim_{h \to 0} \sup_{\mathbf{f} \in \mathcal{B}^{\vec p}} \int_{\R^d} |R_{\rho}\mathbf{f}(x+h)-R_{\rho}\mathbf{f}(x)|^r \d x=0.
\end{align}To this end we will give suitable pointwise bounds on the operator $ R_{\rho}$ and the differences induced by it. Since $M$ is fixed, we will crucially use that $z=(w,t) \in \mathbb B_M$ satisfies
	\begin{equation}\label{e:comp} t \sim 1, \quad |w| \lesssim 1, \quad \chi_z \lesssim \chi,\end{equation} 
where $\chi=\chi_{(0,1)}$ while $\sim$ and $\lesssim$ now denote comparability with constants depending on $M$. Let us now note the preliminary trivial bounds that can be obtained via applying H\"{o}lder's inequality and \eqref{e:comp}. To this end, introduce, for $j=1,\ldots,m-1$,
	\[\lambda_0 = \frac{d}{r} + \eta, \quad \lambda_j = \frac{d}{p_j'} + \eta, \quad \eta = \frac{1}{m}\left( \frac{d}{p_m'}+\delta\right) >0.\]
It is easy to check that $md+\delta=\sum_{j=0}^{m-1} \lambda_j$, $\lambda_0r > d$, and $\lambda_j p_j' > d$, thus for $z \in \mathbb B_M$, 
	\begin{equation}\label{e:prelim} \begin{aligned}
			&\left|\langle f_{m},\phi_z \rangle \right| \leq \left \| f_m \right \|_{p_m} \left \| \phi_z \right \|_{p_m'} \lesssim 1 \\ 
		& \left|\langle v_{z}(x,\cdot),f_1 \otimes \ldots \otimes f_{m-1} \rangle\right| \leq \frac{1}{t^d} \chi_{z}(x)^{\lambda_0}  \prod_{j=1}^{m-1} \left \| f_j \right \|_{p_j} \left \| \frac{1}{t^d} \chi_{z}^{\lambda_j} \right \|_{p_j'} \lesssim \chi(x)^{\lambda_0}.
		\end{aligned} \end{equation}
With these estimates in hand we have that $|R_\rho \mathbf{f}| \lesssim \chi^{\lambda_0}$ and therefore \eqref{e:RK1} holds since $\lambda_0 r>d$. In the same way, we have that
	\[ \left|\langle v_{z}(x,\cdot),f_1 \otimes \ldots \otimes f_{m-1} \rangle -\langle v_{z}(x+h,\cdot),f_1 \otimes \ldots \otimes f_{m-1} \rangle\right| \lesssim |h|^\delta \chi(x)^{\lambda_0} \]
and therefore $|R_\rho \mathbf{f} (x)-R_\rho \mathbf{f}(x-h)| \lesssim |h|^\delta \chi(x)^{\lambda_0}$ from which \eqref{e:RK2} now follows. It remains to handle the endpoint case. {First note that $R_\rho$ is clearly bounded on $\mathcal B^\infty$.} We claim it suffices to establish that for each sequence $\{\mathbf{f}_n\}_{n \in \N} \subset \B^\infty$, $R_\rho \mathbf{f}_n$ has a convergent subsequence in $\CMO$. Indeed, given such an $\{\mathbf{f}_n\}_{n \in \N}$, by a diagonalization argument, we may extract a subsequence $\{\mathbf{f}_{n_k}\}_{k \in \N}$ such that for each $\rho_n = \frac 1n$, $\{ R_{\rho_n} \mathbf{f}_{n_k}\}_{k \in \N}$ is Cauchy in $\CMO$. Then, for any $\epsilon>0$ pick $n$ large enough that
	\begin{equation}\label{e:TRep} \norm{{   (T-R_{\rho_n})[\cdot ]}}_{\infty,\mathrm{BMO}} < \frac{\epsilon}{3}, \end{equation}
recalling the definition \eqref{e:Linfty}.
Such an $n$ exists by the condition $\ep_z \to 0$ and \eqref{e:wave-bound-endpoint}. Then, pick $N$ large enough that for all $i,k  \ge N$, $\norm{R_{\rho_n}\mathbf{f}_{n_i}-R_{\rho_n}\mathbf{f}_{n_k}}_{\BMO} < \frac{\epsilon}{3}$. Therefore, by the triangle inequality, {\eqref{e:TRep}, and recalling that $R_\rho[\mathbf{f_n}]$ and $R_\rho(\mathbf{f}_n)$ only differ by a constant,}
	\[ \norm{T [\mathbf{f}_{n_k}]-T[\mathbf{f}_{n_i}]}_{\BMO} < \epsilon \]
whence $\{T[\mathbf{f}_{n_k}] \}_{k \in \N}$ is Cauchy and has a limit in $\CMO$. Now it remains to show each $R_\rho$ is compact. Applying \eqref{e:prelim} with all $p_j=\infty$ implies the same pointwise estimates as above, which implies $R_\rho:\B^\infty \to \C_{\mathrm{v}}(\mathbb R^d)$ and that the family $\{R_\rho \mathbf{f} : \mathbf{f} \in \B^\infty\}$ is equicontinuous. Therefore, by the Arzela-Ascoli theorem and a diagonalization argument, given a sequence $\{\mathbf{f}_n\}_{n \in \N} \in \B^\infty$, we can obtain a subsequence such that $\{R_\rho \mathbf{f}_{n_k}\}_{k \in\N}$ is Cauchy in $\norm{\cdot}_{L^\infty([-n,n]^d)}$ for each $n \in \N$. However, the pointwise estimate for $R_\rho \mathbf{f}$ shows that given $\epsilon>0$ we can find $n$ large enough that $\abs{R_\rho \mathbf{f}_{n_k}(x)} < \frac{\epsilon}{3}$ for $x$ outside $[-n,n]^d$. Combining these two facts with the triangle inequality shows that $\{R_\rho \mathbf{f}_{n_k}\}_{k \in \N}$ is Cauchy in $\norm{\cdot}_{L^\infty}$ which is a stronger norm than $\norm{\cdot}_{\BMO}$. Finally, recalling that $R_\rho \mathbf{f}_{n_k} \in \C_{\mathrm{v}}(\R^d)$ establishes that the limit belongs to $\CMO$.
\end{proof}

\section{Compactness of paraproducts}\label{s:pp} In this section we will deal with the compactness of the paraproducts that arise from our representation theorem. Specifically, we will prove that the membership of the symbols in $\CMO$ is sufficient for the compactness of the associated paraproduct by giving an essential norm estimate. Given $(\vec p,r) \in Q$, and an $m$-linear operator $T$, define the essential norm
	\[ \left \| T \right \|_{\mathrm{ess}(\vec p,r)}  = \inf_{K \, \mathrm{compact}}
	\left \| T-K\right \|_{\vec p,r},\] 
and the natural modification at the endpoint which we denote by $\norm{\cdot}_{\mathrm{ess}(\infty,\mathrm{BMO})}$ where the operator norm of $T-K$ is measured in $\norm{\cdot}_{\infty,\mathrm{BMO}}$.
\begin{definition}\label{d:pp}
Let $\varu \in \C^\infty_0(B(0,1))$ with $\int \varu(x) \, \d x =1$. For each $z \in Z^d$, set $\varu_z = \Sy_z \varu$. Given $\bb \in \BMO$, define the $(m+1)$-linear form $\Pi_{\bb}$ by
	\[ \Pi_{\bb}(f_0,f_1,\ldots,f_m) = \int_{Z^d} \ip{\bb}{\phi_z}\ip{f_0}{\phi_z} \prod_{j=1}^m \ip{f_j}{\varu_z} \, \d \mu(z).\]
$\Pi_b$ is called a paraproduct form with symbol $\bb$. Any $m$-linear operator $S_{\mathfrak b}$ such that
	\[ \ip{f_0}{S_{\mathfrak b}(f_1,\ldots,f_m)} = \Pi_{\mathfrak b}^\sigma(f_0,\ldots,f_m)\]
for some $\sigma \in \mathrm{S}_{m+1}$ is called a paraproduct with symbol $\bb$.
\end{definition}

The main goal of this section is to prove the following.
\begin{proposition}\label{p:ppc}
If $\bb \in \CMO$, then any paraproduct $S_\bb$ is a compact CZO.
\end{proposition}
{Our proof of Proposition \ref{p:ppc} will argue along the same lines as our proof of Proposition \ref{p:wave-compact} from the previous section. For different line of reasoning, see \cite{BLOT24a}.

To prepare for the proof,} let us review the standard boundedness theory of paraproducts, analogous to Proposition \ref{p:wave-bound} for wavelet forms. To do so, it is convenient to view $\Pi_\bb$ as an $(m+2)$-linear form, where the extra input function is $\bb$ itself. In fact, in this way $\Pi_{\bb}(f_0,\ldots,f_m) = U(\bb,f_0,\ldots,f_m)$ where $U$ is a canonical $(m+2)$-linear wavelet form with $\abs{\ep_z}\lesssim 1$. Such forms are cancellative in the first and second positions, for which a slight strengthening of \eqref{e:wave-bound} and \eqref{e:wave-bound-endpoint} holds: {If $T$ satisfies
	\[ \ip{f_0}{T(f_1,\ldots,f_m)} = U(\bb,f_{\sigma(0)},\ldots,f_{\sigma(m)}) \]
for some $\sigma \in \mathrm{S}_{m+1}$ and canonical wavelet form $U$, then}
	\begin{equation}\label{e:pp-bound} \begin{aligned}  
	&\norm{ T }_{\vec p,r} \le C_{\vec p,r} \norm{\bb}_{\BMO}, \quad \norm{ T }_{\infty,\mathrm{BMO}} \le C_{\infty} \norm{\bb}_{\BMO}, \quad (\vec p, r) \in Q .\end{aligned} \end{equation}
In particular, \eqref{e:pp-bound} applies to $T = S_{\bb}$. 
A proof of \eqref{e:pp-bound} is omitted since it follows from standard considerations; see e.g. the proofs and comments following Propositions 2.5 and 2.7 in \cite{DWW}. 

Now, to give a description of $\CMO$ which is more amenable to $S_{\bb}$, let us introduce an orthonormal wavelet system $ \left\{\psi_I\right\}_{I \in \D}$. Here $\D$ is a dyadic grid on $\mathbb R^d$ and for each $I \in \D$, set $\zeta(I)=(c(I),\ell(I)) \in Z^d$, where $c(I)$ is center of the cube $I$ and $\ell(I)$ the side length. Then $\psi_I = \Sy_{\zeta(I)} \psi$ for a specific $\psi \in C\Psi^{1,1;1}_{(0,1)}$. We have kept $\psi_I$ to be $L^1$-normalized, so the reproducing formula is
	\[ f = \sum_{I \in \D} |I| \ip{f}{\psi_I} \psi_I.\]
Now introduce the family of orthogonal projections for $M \ge 1$,
\[P_Mf= \sum_{I \in \D_M}|I| \langle f,\psi_I \rangle \psi_I, \quad  
	\mathcal{D}_M \coloneqq \left\{ I \in \D :  \zeta(I) \in \mathbb B_M \right\}, \quad P_M^\perp = \mathrm{Id}-P_M .\]
From \cite{Vil15}*{Lemma 2.20}, an equivalent characterization of $\bb \in \CMO$ is that
	\begin{equation}\label{e:CMO-PM} \lim_{M \to \infty} \norm{P^\perp_M \bb}_{\BMO} =0.\end{equation}
Therefore, Proposition \ref{p:ppc} will be a consequence of the following essential norm estimate.

\begin{proposition} \label{paraprodess} Let $S_{\bb}$  be a paraproduct with symbol $\bb \in \BMO$. Then for each $(\vec p,r) \in Q \cup \{ (\infty,\mathrm{BMO})\}$, \begin{equation}\label{e:pp-ess} \left \| S_{\bb} \right \|_{\mathrm{ess}(\vec p,r)} \lesssim \liminf_{M\to \infty} \left \| P_{M}^{\perp} \bb \right \|_{\BMO}.\end{equation}\end{proposition}
	\begin{proof} Let $M \ge 1$ large, and perform the splitting $S_{\bb} = R_M + T_M$ where 
	\[ T_M (f_1,\ldots,f_m) = \int_{Z^d \setminus \mathbb  B_{100 M}} \ip{\bb}{\phi_z} \prod_{j=1}^m \ip{f_j}{\varu_z} \phi_z \, \d\mu(z),\]
and $R_M$ is same but the integration is taken over $\mathbb B_{100 M}$. The operator $R_{M}$ is clearly compact by repeating the discussion made in \S \ref{s:compact} to show that $R_\rho$ there was compact. Therefore the essential norm of $S_{\bb}$ is controlled by the operator norm of $T_M.$ Now,  for $z \not \in \mathbb{B}_{100 M},$ we can calculate the pairings appearing in the above equation by expanding $\bb$ with the aid of the wavelet basis $ \left\{\psi_I\right\}_{I \in \D}$: \[  \langle \bb, \phi_{z} \rangle = \langle P_M\bb, \phi_{z} \rangle+ \langle P_{M}^{\perp}\bb, \phi_{z} \rangle. \] For the first term we will employ the linear wavelet averaging process from Lemma \ref{l:avg}. We expand 
	\[ \langle P_M\bb, \phi_{z} \rangle= \sum_{I \in \D_M} |I| \langle \bb, \psi_I \rangle \langle \psi_I, \phi_{z} \rangle= \left  \langle \bb, \sum_{I \in \D_M}|I| \langle \psi_I, \phi_{z} \rangle \psi_I \right \rangle.  \] Since $\psi_I \in \Psi^{1,1;1}_{\zeta(I)}$ and $\phi_z \in \Psi^{1,1;1}_z$, one can compute $\left| \langle \psi_I, \phi_{z} \rangle\right| \lesssim [z,\zeta(I)]_{\frac 12}$ (see e.g. \cite{DWW}*{Lemma 2.3} or \cite{FJW}*{Appendix, Lemmata 2 and 4}). Furthermore, since $\zeta(I) \in \mathbb B_M$ and $z \not\in \mathbb B_{100 M}$, $[z,\zeta]_{\frac 12} \lesssim M^{-\frac 14} [z,\zeta]_{\frac 14}$. Now, to apply Lemma \ref{l:avg}, rewrite
	\[ \sum_{I \in \D_M}|I| \langle \psi_I, \phi_{z} \rangle \psi_I = \int_{Z^d} H(\zeta,z) \tilde \psi_\zeta \, \d\mu(\zeta), \quad \]
where for each $I \in \D_M$ and $\zeta = (w,t) \in I \times (\tfrac{\ell(I)}{2},\ell(I)]$, we define
	\[ \tilde \psi_\zeta = \psi_I, \quad H(\zeta,z) = \frac{|I|}{\mu(I \times (\tfrac{\ell(I)}{2},\ell(I)])} \langle \psi_I, \phi_{z} \rangle,\]
and $H(\zeta,z)=0$ if $\zeta \not \in \cup_{I \in \D_M} I \times (\tfrac{\ell(I)}{2},\ell(I)]$. Since $\abs{H(\zeta,z)} \lesssim M^{-\frac 14} [z,\zeta]_{\frac 14}$, Lemma \ref{l:avg} provides a universal constant $C$ and $\lambda_z \in C \Psi_{z}^{1,\frac 18;1}$ such that 
\[T_M (f_1,\ldots,f_m) = \int_{Z^d \setminus \mathbb  B_{100M}} \left ( M^{-\frac 14} \ip{\bb}{\lambda_z} + \ip{P_M^\perp \bb }{\phi_z} \right) \prod_{j=1}^m \ip{f_j}{\varu_z} \phi_z \, d\mu(z). \]
Therefore, \eqref{e:pp-ess} follows by the triangle inequality and \eqref{e:pp-bound}.
\end{proof}

\section{Main result and proof}\label{s:all}
Let us now put together the pieces from the previous sections to prove the following compact $T(1)$ wavelet representation theorem.
\begin{theorem}\label{thm:all}
Suppose $T$ is an $m$-linear CZO, $\Lambda$ the associated $(m+1)$-linear form, with $\bb_j \in \BMO$, $j=0,1,\ldots,m$ satisfying \eqref{e:T1}. The following are equivalent.
\begin{itemize}
	{   \item[A.]$T$ is a compact CZO in the reflexive range.
	\item[B.] For any $\varphi^1,\ldots,\varphi^m \in L^\infty$ with compact support, and every $j=0,1,\ldots,m$,
	\[ \lim_{z \to \infty} t^{md-\frac d2}\norm{T^{*,j} (\varphi ^1_z,\ldots,\varphi ^m_z)}_{L^2(\mathbb R^d)}=0, \quad \varphi_z^i = \Sy_z \varphi^i.\]}
	\item[C.] $\Lambda$ is a compact CZ form, i.e. $\mathsf{W}_\Lambda^M(z) \to 0$ as $z \to \infty$ and $\bb_j \in \CMO$.
	\item[D.] There exist compact wavelet forms $\{U_k\}_{k=1}^{K_m}$ and compact paraproduct forms $\{\Pi_{\bb_j}\}_{j=0}^m$ such that for all $f_j \in \Schw(\R^d)$,
	\[ \Lambda(\mathbf f) = \sum_{k=1}^{K_m} U_k(\mathbf f) + \sum_{j=0}^m \Pi^{*,j}_{\bb_j}(\mathbf f), \quad \mathbf f = (f_0,\ldots,f_m).\]
	\item[E.] $T^{*,j}$ is a compact CZO for each $j=0,1,\ldots,m$.
\end{itemize}
\end{theorem}
\begin{proof}
{It is trivial that E. implies A., and the implication from A. to B. and B. to C. is precisely the content of Proposition \ref{p:nec}.}
To show C. implies D. we isolate the cancellative part of $\Lambda$, namely
	\[ \Lambda_{\mathsf{c}}=\Lambda- \Pi_\Lambda, \quad \Pi_\Lambda = \sum_{j=0}^m \Pi^{*,j}_{\bb_j}.\] 
$\Lambda_{\mathsf{c}}$ is definitely a CZ form, and to verify that it is cancellative, simply note that since $\ip{\varu}{1} =1$, by the reproducing formula \eqref{e:calderon}, for $\varphi \in \Schw^0(\R^d)$,
	\[ \Pi_{\bb}(\varphi,1,\ldots,1) = \int_{Z^d} \ip{\bb}{\phi_z} \ip{\phi_z}{\varphi} \, \d\mu(z) = \ip{\bb}{\varphi}, \]
and since $\phi$ is cancellative, $\Pi^{*,j}_{\bb}(\varphi,1,\ldots,1)=0$  for $j =1,\ldots,m$. We want to apply Proposition \ref{p:rep} to $\Lambda_{\mathsf{c}}$ so we must establish that it is a compact CZ form. Since $\bb_j \in \CMO$, by {Remarks \ref{r:cz} and \ref{r:ex},} $\Pi_\Lambda$ is a compact CZ form. Since we also know that $\Lambda$ is a compact CZ form, $\Lambda_{\mathsf{c}}$ must indeed be compact, and B. follows by applying Proposition \ref{p:rep} to $\Lambda_{\mathsf{c}}$. D. implies E. is a consequence of Propositions \ref{p:wave-compact} and \ref{p:ppc}.\end{proof}

A sixth equivalent condition in Theorem \ref{thm:all} would be the compactness of the adjoint operators on \textit{weighted} Lebesgue spaces for weights belonging to the multilinear Muckenhoupt classes \cites{lerner-et-al,nieraeth2019quantitative}. If one is only concerned with weights corresponding to $p_j<\infty$, then this follows immediately from \cite{COY}. Alternatively, to handle the full range of $Q$, including when some $p_j=\infty$, one can verify the conditions of the weighted Riesz-Kolmogorov theorem \cite{COY}*{Proposition 2.9} as we did in Propositions \ref{p:wave-compact} and \ref{p:ppc} by hand.

{Furthermore, let us briefly recall the notion of a compact CZ kernel, which has been ubiquitous in the compact CZO literature until \cite{MS23} recently removed this condition. A compact CZ kernel satisfies the kernel estimates \eqref{e:ker-size} and \eqref{e:ker-smooth}, but in addition to the uniform bound provided by $C_K$, one imposes an additional prefactor on the RHS,
	\[  o\left(\max\left\{\abs{x_0-\mathbf{x}}_2,\abs{x_0-\mathbf{x}}_2 ^{-1}, \abs{x_0+\mathbf{x}}_2 ^{-1}\right\}\right).\] 
By standard arguments \cite{tao-notes}*{Lecture 7}, if one assumes $T$ has a compact CZ kernel and satisfies the \textit{diagonal} weak compactness property
	\begin{equation}\label{e:diag} \lim_{z\to \infty} \Lambda(\mathbf 1_{B(w,t)},\ldots,\mathbf 1_{B(w,t)}) =0,\end{equation}
then condition B. in Theorem \ref{thm:all} holds from whence all the other conclusions follow. This perspective, using the compact kernel and \eqref{e:diag}, was recently undertaken in \cite{CLSY24} in the bilinear setting.}

We finally note that Theorem \ref{thm:all} applies to the linear case as well and a few simplifications can be made due to the greater symmetry enjoyed in this setting. Let us restate Theorem \ref{thm:all} when $m=1$ for additional clarity.

\begin{theorem}\label{thm:linear}
Let $T$ be a linear CZO. The following are equivalent. 
	\begin{itemize}
	{   
	\item[A.] $T$ is compact on $L^2(\mathbb R^d)$.
	\item[B.] For any $\varphi \in L^\infty$ with compact support,
		\[ \lim_{z\to\infty} t^{\frac d2} \norm{T(\varphi_z)}_{L^2} + t^{\frac d2} \norm{T^*(\varphi_z)}_{L^2} = 0,  \quad \varphi_z = \Sy_z \varphi. \]}
	\item[C.] $T(1), T^*(1) \in \CMO$ and 
	\[ \lim_{\zeta \to \infty} \sup_{z \in \mathbb B(\zeta)} t^d \abs{ \ip{T \phi_z}{\phi_\zeta} } =0.\]
	\item[D.] There exists a compact wavelet form $U$ and compact paraproduct forms $\Pi_{T(1)},\Pi_{T^*(1)}$ such that for all $f,g \in \Schw(\R^d)$,
	\[ \ip{Tf}{g} = U(f,g) + \Pi_{T(1)}(f,g) + \Pi_{T^*(1)}(g,f).\]
	\item[E.] $T$ and $T^*$ are compact CZOs.
	\end{itemize}
\end{theorem}
\begin{proof}
Only A. and C. differ from their multilinear statements in Theorem \ref{thm:all}. To compare C. with its $m$-linear analogue, one only needs to note that $z \in \mathbb B_M(\zeta)$ is equivalent to $\zeta \in \mathbb B_{M'}(z)$ as long as $M \sim M'$. Therefore, the implications from B. to C., from C. to D., and from D. to E. all follow from Theorem \ref{thm:all}. E. trivially implies A., so it remains to establish A. implies B. where the true force of the linear situation is used. {Namely, $t^{\frac d2}\phi_z$ converges weakly to zero and therefore, since $T$ is compact $TT^*(t^{\frac d2}\phi_z)$ and $T^*T (t^{\frac d2}\phi_z)$ both converge strongly to zero, which implies B.}
\end{proof}
{   
Finally, we conclude with a short, simple application.
\subsection*{Example}
Given functions $a,b \in L^\infty(\mathbb R^d)$, consider the operators of multiplication on the Fourier and space sides by
	\[ T_a f (x) = \int_{\mathbb R^d} a(\xi) \hat f(\xi) e^{i x \cdot \xi} \, d\xi, \quad M_b f(x) = b(x) f(x).\]
It is easy to see that unless $a$ or $b$ is identically zero, such operators can never be compact. This even persists for the multilinear analogues, in particular the $m$-linear Fourier multiplier
	\[ T_a(f_1,\ldots,f_m)(x) = \int_{\mathbb R^{md}} a(\xi) \prod_{j=1}^m \hat f_j(\xi_j) e^{i x(\xi_1+\cdots + \xi_m)} \, d \xi, \quad \xi = (\xi_1,\ldots,\xi_m), \xi_j \in \mathbb R^d.\] 
Nonetheless, by the uncertainty principle, we expect compositions of $T_a$ and $M_b$ to be compact. The most well-known example is in the linear case when $a,b \in L^2(\mathbb R^d)$. Then trivially $T_aM_b$ is Hilbert-Schmidt under no smoothness assumptions on $a$ and $b$.

We must impose some conditions on $a$ and functions $b_0,\ldots,b_m$ so that suitable compositions of $T_a$ and $M_{b_j}$ are CZOs. Assume there exists $\delta>0$ such that $b_j \in \C^\delta(\mathbb R^d)$ and $a$ belongs to $\C^\infty(\mathbb R^{md} \setminus\{0\})$ such that
	\begin{equation}\label{e:hormander} \abs{\partial^\alpha a(\xi)}\leq C_\alpha \vert\xi\vert^{-|\alpha| }, \quad \abs{b_j(x)} + \frac{\abs{b_j(x) - b_j(x+h)}}{\abs{h}^\delta} \lesssim (1+\abs{x})^{-\delta} .\end{equation}
It is a routine exercise to check that under these conditions, the operator
	\[ S(f_1,\ldots,f_m) = M_{b_0} T_a(M_{b_1}f_1,\ldots,M_{b_m}f_m) \]
is an $m$-linear CZO satisfying $\mathfrak b_j \in \CMO$. Therefore, by Theorem \ref{thm:all}, compactness of $S$ is characterized by $\mathsf{W}_\Lambda(z) \to 0$. To check that, we examine quantities of the form
	\begin{equation}\label{e:WL-ver} \int_{\mathbb R^{md}} a(\xi_1,\ldots,\xi_m) \widehat{c_0g^0_{z_0}}(\xi_1+\cdots+\xi_m) \prod_{j=1}^m \widehat{c_j g^j_{z_j}}(\xi_j) \, d\xi, \quad z_j \in \mathbb B(z),\end{equation}
where $(c_0,\ldots,c_m)$ is some permutation of $(b_0,\ldots,b_m)$ and each $g^i$ is either $\phi$, $\psi^1$, or $\psi^3$. One can rewrite, for any $\eta \in \mathbb R^d$,
	\[ \widehat{c_j g^j_{z_j}}(\eta) = \ip{ e^{i \eta \cdot} c_j}{g^j_{z_j}}\]
to see that for each $\xi$, the integrand in \eqref{e:WL-ver} goes to zero as $z \to \infty$. Then, if $a \in L^1(\mathbb R^{md})$, by dominated convergence, the entire quantity \eqref{e:WL-ver} goes to zero, and $S$ is a compact CZO.}

\bibliography{BilCptT1}
\bibliographystyle{amsplain}
	
\end{document}